\documentclass[11pt]{amsart}
\usepackage{natbib}

\newcommand{\Pp}{\mathcal{P}}
\newcommand{\Z}{\mathbb {Z}}
\newcommand{\qbin}[2]{\genfrac{[}{]}{0pt}{}{#1}{#2}}
\newcommand{\gp}[3]{\qbin{#1}{#2}_{#3}}
\newcommand{\Bin}{\mathrm{Bin}}

\newcommand{\Cov}{\mathrm{Cov}}
\newcommand{\inv}{\mathrm{inv}}

\newtheorem{theorem}{Theorem}
\newtheorem{corollary}[theorem]{Corollary}

\newtheorem{lemma}[theorem]{Lemma}

\theoremstyle{remark}
\newtheorem{remark}[theorem]{Remark}

\numberwithin{theorem}{section}
\numberwithin{equation}{section}

\allowdisplaybreaks

\title[A Refinement of the Binomial Distribution]{A Refinement of the Binomial Distribution\\Using the 
Quantum Binomial Theorem}
\author{Andrew V. Sills}
\email{asills@georgiasouthern.edu}
\address{Department of Mathematical Sciences,
Georgia Southern University, Statesboro and Savannah, Georgia, USA}
\date{\today}

\begin{document}
\maketitle

\begin{abstract}
$q$-analogs of special functions, including hypergeometric functions,
play a central role in mathematics and have numerous applications in
physics.  In the theory of probability, $q$-analogs of various probability
distributions have been introduced over the years, including the binomial
distribution.  Here, I propose a new refinement of the binomial distribution
by way of the quantum binomial theorem (also known as the 
the noncommutative $q$-binomial 
theorem), where the $q$ is a formal variable in which information
related to the sequence of successes and failures in the
underlying binomial experiment is encoded 
in its exponent. 
\end{abstract}

\section{Background and motivation}
Many of the standard mathematical objects used in probability and statistics 
(e.g. factorials, the gamma function, the exponential function, 
the beta function, the binomial
co\"efficients, etc.) have well-known $q$-analogs that play a central r\^ole
in the theory of special functions~\citep[see][]{AAR99} and basic hypergeometric series
~\citep[see][]{GR04}.
A $q$-analog of a mathematical object $A$ is a function $f(q)$ such that
$f(1) = A$ or, failing that, at least
$\lim_{q\to 1} f(q) = A$ and in some imprecise sense, $f(q)$ retains some of the
remarkable properties possessed by $A$.


A number of $q$-analogs of the binomial distribution have been introduced
over the years,~\citep[see][]{D81,K87,KK91,S94,K02,K12}.  All of these
are based on the standard commutative version of the $q$-binomial
theorem~\citep[see, e.g.,][p. 488, Theorem 10.2.1]{AAR99}.

Here we propose a refinement of the binomial
distribution that draws its inspiration from the \emph{quantum}
binomial theorem, also known as the 
\emph{noncommutative} $q$-binomial
theorem~\citep{P50,S53}.  It retains the ordinary binomial experiment
setting (counting successes in $n$ independent Bernoulli trials), but the 
inclusion of the additional parameter $q$ encodes additional 
combinatorial information.  Unlike the $q$-analogs 
of various probability distributions well known in the literature, where $q$ can
meaningfully assume a numerical value within a specified range, the
$q$ presented herein is strictly a \emph{formal variable} with
associated combinatorial 
information encoded in its exponent. 
The classical binomial
probability distribution is recovered when the formal variable $q$ is set equal to unity
in our construction.

   In section 2, we recall some basic facts about integer partitions, state and prove
a generalization of the quantum binomial theorem, and then derive the 
quantum binomial theorem and the classical binomial theorem as corollaries.
In section 3, the binomial probability distribution is discussed, and a refinement
inspired by the quantum binomial theorem is motivated.  In section 4, the
exponent of $q$ from the previous section is interpreted in its new r\^ole as a
random variable.   In section 5, our refined binomial distribution is given,
where we now consider a joint probability distribution of two random variables
describing the original setting of $n$ independent Bernoulli trials with constant
probability of success $\pi$ on each trial.  Marginal and conditional distributions
are derived along with various moments.  In section 6, some concluding words
are offered, suggesting possibilities for further research.

\section{Introduction}
Let 
\begin{multline*}
\Pp_{k,m}  := \\ \{ (\lambda_1, \lambda_2, \dots, \lambda_k) :
  m \geq \lambda_1 \geq \lambda_2 \geq \cdots \geq \lambda_k \geq 0 
  \mbox{ and each $\lambda_j \in \Z$} \}. 
\end{multline*}
Thus $\Pp_{k,m}$ can be thought of as the set of all partitions with at most
$k$ parts, none of which is greater than $m$.  If a given partition has 
strictly less than $k$
parts, we simply pad on the right with zeros. 
For example, 
\begin{multline*}
  \Pp_{3,2} = \{ (0,0,0), (1,0,0), (1,1,0), (2,0,0),  \\
  (1,1,1), (2,1,0), (2,1,1), (2,2,0),
  (2,2,1), (2,2,2)\} .
\end{multline*}
Equivalently, $\Pp_{k,m}$ can be visualized as the
set of all Ferrers diagrams that fit inside a rectangle $k$ units high and $m$ units
wide~\citep[see][p. 6 ff]{A76}.  
If $\lambda = (\lambda_1, \lambda_2, \dots, \lambda_k)$, we let
$|\lambda| := \lambda_1 + \lambda_2 + \cdots + \lambda_k$ and call this the
\emph{size} (or \emph{weight}) of $\lambda$.  Notice that $\Pp_{k,m}$ consists of
partitions with sizes from $0$ to $mk$ inclusive.  Later, we will need to consider
only those partitions of size $t$ in $\Pp_{k,m}$, and note this set as $\Pp_{k,m}(t)$.

 Let $x$ and $y$ be indeterminates that do not commute under multiplication.
For $\lambda = (\lambda_1, \dots, \lambda_k) \in\Pp_{k,n-k}$, let $Q^\lambda$  denote the \textbf{operator} that permutes
the factors of $x^{n-k} y^k$ by the product of transpositions 
\begin{equation} Q^\lambda:=\prod_{j=1}^k (n-k+j, n-k+j - \lambda_j). \label{prodtrans} \end{equation}
The transposition, written in cycle notation, 
$(i,j)$ applied to $x^{n-k} y^k$ means we swap the $i$th and $j$th
factors of $x^{n-k} y^k$.  The transpositions are not in general 
disjoint and therefore their product
is not commutative.  We interpret the order of factors in~\eqref{prodtrans} as
the $j=1$ factor is applied first (rightmost), 
the $j=2$ factor is applied second (immediately to the left of
the $j=1$ factor), etc.

For example, with $n-k = 5$, $k=3$, we have
\begin{align*}
Q^{(3,1,0)} x^5 y^3 & =  (8,8)(7,6)(6,3) (x x x x x y y y) \\
                             & = (8,8)(7,6) (x x y x x x y y )\\
                             & = (8,8) (x x y x x y x y) \\
                            & =  x x y x x y x y . 
\end{align*}   

  For a bijection between the partition $\lambda\in\Pp_{k,n-k}$ and the 
permutation of $x^{n-k} y^k$ 
represented by $Q^\lambda x^{n-k} y^k$, see~\citet[p. 40, Theorem 3.5]{A76}.

\begin{theorem}[generalized quantum binomial theorem]\label{gqbt}
For non-commuting indeterminates $x$ and $y$ and the operator $Q^{\lambda}$
defined in~\eqref{prodtrans},
  \begin{equation} \label{gqbteq}
    (x+y)^n = \sum_{k=0}^n \sum_{\lambda\in\Pp_{k,n-k} } Q^{\lambda} x^{n-k} y^k .
  \end{equation}  
\end{theorem}

Note that the operator $Q^\lambda$ acts as a generalization of the 
formal expression $q^{|\lambda|}$.
Before proving Theorem~\ref{gqbt}, we will provide some context and
motivation.
If we replace the operator $Q^{\lambda}$ with the $|\lambda|$th power
of the indeterminate $q$ (note that $q$ commutes with both $x$ and $y$), we
obtain the quantum binomial theorem\footnote{The 
name ``quantum binomial theorem'' dates back to
at least~\citet{BDR02}.  Elsewhere in the literature (see, e.g., \citet[p. 485]{AAR99}) 
this result is called the
``noncommutative $q$-binomial theorem''.  Other authors just call it the
``$q$-binomial theorem,'' but this risks possible confusion with other 
(commutative) 
results also known by that name.}, usually attributed to
M. P.~\citet{S53}, but note also an essentially equivalent result
due to H. S. A.~\citet{P50}.   To state this theorem, we need the usual
$q$-binomial co\"efficient, also known as the \emph{Gaussian polynomial}:
\begin{equation} \label{gp}
 \gp nkq := \frac{(1-q^{n-k+1})(1-q^{n-k+2}) \cdots (1-q^{n})}{(1-q)(1-q^2)\cdots
  (1-q^k)},
\end{equation} for $0\leq k \leq n$, and $0$ otherwise.  It is well known that
\eqref{gp} is a polynomial in $q$ of degree $k(n-k)$, satisfies $q$-analogs
of the Pascal triangle recurrence, and is the generating function for the function
that counts the number of members of
$\Pp_{k,n-k}$ ~\cite[see][Chapter 3]{A76} of a given size:

\begin{equation}\label{key}
  \gp nkq = \sum_{\lambda\in\Pp_{k,n-k}} q^{|\lambda|} =
  \sum_{j=0}^{k(n-k)} \#\Pp_{k,n-k}(j) q^j,
\end{equation}
where $\#\Pp_{k,m}(j)$ is the number of partitions of size $j$ into at most $k$ parts, 
with each part at most $m$.

Recall that the Potter--Sch\"utzenberger quantum binomial
theorem may be stated as follows:
\begin{corollary}[Quantum binomial theorem] \label{ncqbt}
If $yx = qxy$, and $n$ is a nonnegative integer, then
\begin{equation} \label{ncqbteq}  
(x+y)^n = \sum_{k=0}^n \gp nkq x^{n-k} y^k. 
\end{equation}
\end{corollary}
To obtain Corollary~\ref{ncqbt} from Theorem~\ref{gqbt}, we replace
the operator $Q^\lambda$ with the formal variable $q^{|\lambda|}$ and
use~\eqref{key}.

Of course, the $q=1$ case of Corollary~\ref{ncqbt} 
(so that multiplication of $x$ and $y$ is now commutative)
is the classical binomial theorem:
\begin{corollary}[classical binomial theorem] For nonnegative integer $n$,
\begin{equation} \label{bteq}  
 (x+y)^n = \sum_{k=0}^n \binom nk x^{n-k} y^k.
\end{equation}

\begin{proof}[Proof of Theorem~\ref{gqbt}]
Observe that Eq.~\eqref{key} provides the key link between
Theorem~\ref{gqbt} and Corollary~\ref{ncqbt}; in fact the first equality
of~\eqref{key} is all we need to establish Theorem~\ref{gqbt} given
Corollary~\ref{ncqbt}, the result of Potter--Sch\"utzenberger.
\end{proof}

\end{corollary}

\section{On the binomial distribution}
In the binomial experiment, we have $n$ independent Bernoulli
trials where the probability of a ``success'' on each trial 
is some fixed value $\pi$, $0<\pi<1$, and
the probability of a ``failure'' is $1-\pi.$
\begin{remark}{We acknowledge immediately that there is an
unfortunate conflict
between the standard notations of probability and that of the theory of $q$-series.
Accordingly, here we \emph{avoid} using $p$ for the probability of 
a success a Bernoulli trial (since $p(n)$ is a standard notation for the number of partitions
of $n$) and avoid $q = 1-p$ for the probability of a failure
on a Bernoulli trial, as ``$q$'' is used in the sense of $q$-series. 
}
\end{remark} 

If $Y$ is the random variable that counts
the number of successes encountered during the $n$ 
independent Bernoulli trials, then
we say $Y$ is a binomial random variable with parameters $n$ and $\pi$, 
writing $Y\sim\Bin(n,\pi)$ for short.  Note that 
\[ P(Y=k) = \binom nk (1-\pi)^{n-k} {\pi}^k, \]  for $k=0,1,2,\dots, n$, and that 
\[ \sum_{k=0}^n P(Y=k) = 1 \] 
follows immediately by taking $x=1-\pi$ and $y=\pi$ in~\eqref{bteq}.

Notice, however, that if we apply the same interpretation of $x$ and $y$
(i.e. probabilities of failure and success respectively) to the context of 
Theorem~\ref{gqbt}, then we are using the extra information preserved to
track each possible sequence of $n$ Bernoulli trials.  For example, consider
the $n=4$ case.  
There is a one-to-one correspondence between terms generated by the
left member of~\eqref{gqbteq} and those generated by the right member as
follows, where F denotes failure and S denotes success, grouped by the
values of $Y = 0, 1, 2, 3, 4$:

\begin{center}
\begin{tabular}{| c |c| c|}
\hline
Outcome & LHS summand & corresponding RHS summand\\
\hline\hline
FFFF & $xxxx$ & $Q^{\emptyset} x^4 = () x^4$\\
\hline
FFFS & $xxxy$ & $Q^{(0)} x^3 y = (4,4) x^3 y $\\
FFSF & $xxyx$ & $Q^{(1)} x^3 y = (4,3) x^3 y$\\
FSFF & $xyxx$ & $Q^{(2)} x^3 y = (4,2) x^3 y$\\
SFFF & $yxxx$ & $Q^{(3)} x^3 y = (4,1) x^3 y$ \\
\hline
FFSS & $xxyy$ & $Q^{(0,0)} x^2 y^2 = (4,4)(3,3) x^2 y^2 $\\
FSFS & $xyxy$ & $Q^{(1,0)} x^2 y^2 = (4,4)(3,2) x^2 y^2$\\
SFFS & $yxxy$ & $Q^{(2,0)} x^2 y^2 = (4,4)(3,1) x^2 y^2$\\
FSSF & $xyyx$ & $Q^{(1,1)} x^2 y^2 = (4,3)(3,2) x^2 y^2$\\
SFSF & $yxyx$ & $Q^{(2,1)} x^2 y^2 =  (4,3)(3,1) x^2 y^2$ \\
SSFF & $yyxx$ & $Q^{(2,2)} x^2 y^2 =   (4,2)(3,1) x^2 y^2$\\
\hline
FSSS & $xyyy$ & $Q^{(0,0,0)} xy^3 = (4,4)(3,3)(2,2)x y^3 $\\
SFSS & $yxyy$ & $Q^{(1,0,0)} xy^3 = (4,4)(3,3)(2,1) xy^3$ \\
SSFS & $yyxy$ & $Q^{(1,1,0)} xy^3 = (4,4)(3,2)(2,1) xy^3$ \\
SSSF & $yyyx$ & $Q^{(1,1,1)} xy^3 = (4,3)(3,2)(2,1) xy^3$\\
\hline
SSSS & $yyyy$ & $Q^{(0,0,0,0)} y^4 = (4,4)(3,3)(2,2)(1,1) y^4$
\\ \hline \hline
\end{tabular}
\end{center}

If we consider the binomial experiment from the perspective of 
Corollary~\ref{ncqbt}, we have more information than in the ordinary 
binomial distribution, but not always enough to uniquely identify each summand in
the right member of~\eqref{ncqbteq} with a specific sequence of successes and
failures.  (Notice, e.g., in the table below that
outcomes SFFS and FSSF both contribute a factor of $q^2 x^2 y^2$ to
the sum, and thus cannot be distinguished at this level of refinement.
To remedy this, we proposed Theorem~\ref{gqbt}.) 
 
 Once again, consider the $n=4$ case in detail:
\begin{center}
\begin{tabular}{| c | c |c| c |}
\hline
& & & corresponding\\
$Y$ & Outcome & LHS summand & RHS summand\\
\hline\hline
0 &FFFF & $xxxx$ & $ x^4$\\
\hline
&FFFS & $xxxy$ & $ x^3 y $\\
1&FFSF & $xxyx$ & $q x^3 y$\\
&FSFF & $xyxx$ & $q^{2} x^3 y $\\
&SFFF & $yxxx$ & $q^3 x^3 y$ \\
\hline
&FFSS & $xxyy$ & $x^2 y^2  $\\
&FSFS & $xyxy$ & $q x^2 y^2 $\\
2&SFFS & $yxxy$ & $q^2 x^2 y^2 $\\
&FSSF & $xyyx$ & $q^2 x^2 y^2 $\\
&SFSF & $yxyx$ & $q^3 x^2 y^2$  \\
&SSFF & $yyxx$ & $q^4 x^2 y^2$\\
\hline
&FSSS & $xyyy$ & $xy^3$\\
3&SFSS & $yxyy$ & $q xy^3$ \\
&SSFS & $yyxy$ & $q^2 xy^3$ \\
&SSSF & $yyyx$ & $q^3 xy^3 $\\
\hline
4&SSSS & $yyyy$ & $y^4$
\\ \hline \hline
\end{tabular}
\end{center}

\begin{remark}
Note that the noncommutivity of $x$ and $y$ is essential here: 
the noncommutivity is that which allows us
track information relating to the where the successes and failures occur in 
the underlying
binomial experiment.
\end{remark}

Motivated by the similarity in appearance between the pmf for $Y\sim\Bin(n,\pi)$,
\[ P(Y=k) = \binom nk (1-\pi)^{n-k} \pi^k, \] which occurs as the generic summand
in right member of~\eqref{bteq} with $x=1-\pi$ and $y=\pi$, and 
the expression \[ \gp nkq (1-\pi)^{n-k} \pi^k, \] we examine whether the latter
(which is the generic summand in the right member of~\eqref{ncqbteq} with
$x$ replaced by $1-\pi$ and $y$ replaced by $\pi$) can be utilized to generalize the binomial distribution in
some useful sense.

Let us informally define a
``$q$-generalized probability mass function'' $P_q(Y = k)$ as one where
the sum over its support is in general not $1$, but rather a $q$-analog of $1$,
i.e. a function of $q$ that evaluates to $1$ when $q = 1$.   Additionally, 
although we have already emphasized the role of $q$ as a formal variable, we
could safely impose the additional condition that $P_q(Y=k) \geq 0$ for real
$q$, $0 < q < 1$, to mimic the classical condition that a probability mass
function must be nonnegative.

For the case at hand, we wish to have
\begin{equation} \label{qpmf}
P_q(Y = k) = \gp nkq (1-\pi)^{n-k} \pi^k . \end{equation}
Again, the expression $P_q(Y=k)$
\emph{fails to retain} the property from
the classical $q=1$ case that summing over
the support yields unity.  In fact, 
\[ \sum_{k=0}^n P_q(Y=k) = 
\sum_{k=0}^n \gp nkq (1-\pi)^{n-k} \pi^k \] is a polynomial $s(q)$ in $q$ with
the property that $s(1) = 1$.
Also, 
\[ \gp nkq (1-\pi)^{n-k} \pi^k \geq 0 \] for all $0<q<1$, although we prefer
to not evaluate the preceding expression for different values of $q$, 
with the exception of $q=1$, and then only when we wish to pass from
the proposed $q$-generalization to the classical case.

  For example, in the $n=4$ case, 
\begin{equation} \label{PqY2} 
P_q(Y=2) = (1-\pi)^2 \pi^2 (1 + q + 2 q^2 + q^3 + q^4) .
\end{equation}
If we set $q=1$, we recover the fact that if $Y\sim\Bin(4,\pi)$, 
$P(Y=2) = 6(1-\pi)^2 \pi^2$.
However, if we leave the $q$ unevaluated, the polynomial $P_q(Y=2)$ effectively
segregates the $\binom 42 = 6$ outcomes with $k = 2$ successes and 
$n-k = 2$ failures into
$k(n-k) + 1 = 5$ subcategories according to where the $n-k = 2$ failures occur in 
relation to the $k = 2$ successes in the sequence of $n = 4$ trials.  
This phenomenon will be explored further in the next section.


\begin{remark}
Another connection with the $q$-series literature is as follows.  The
\emph{Rogers--Szeg\H{o} polynomials} $H_n(q; z)$,
(see, e.g., \citet[pp. 49--50]{A76}), a family of
polynomials orthogonal on the unit circle, are defined as
\begin{equation} \label{RSP}
  H_n(q;z) := \sum_{k=0}^n \gp nkq z^k.
\end{equation}
Thus the sum over the support of the $q$-generalized pmf is
\begin{equation} \label{RSsum}  \sum_{k=0}^n P_q(Y=k) = 
\sum_{k=0}^n \gp nkq (1-\pi)^{n-k} \pi^k = (1-\pi)^n H_n\left(q; \frac{\pi}{1-\pi}\right). 
\end{equation}
\end{remark}
 An anonymous referee pointed out that by dividing through by the right member of
\eqref{RSsum}, one would obtain a legitimate pmf, i.e. for fixed parameters
$n$, $\pi$, and $q$,
let 
\[ P(Y=k) = \frac{  \gp nkq (1-\pi)^{n-k} \pi^k }{ (1-\pi)^n H_n(q; \pi/(1-\pi))} =
\gp nkq \frac{ \pi^k }{ (1-\pi)^k H_n (q; \pi/(1-\pi)) } ,\] and then it must be the
case that \[ \sum_{k=0}^n P(Y=k) = 1.  \]  However, the referee went on to observe
that ``[u]nfortunately, this simple rectification ruined the entire model; the random
variable $Y$ does not count, anymore, the number of successes in a sequence
of $n$ independent Bernoulli trials, with constant success probability.''

\section{Interpretation of the exponent of $q$}

A precise interpretation of the term $q^t x^{n-k} y^k$ may be given as follows:
consider an outcome of a binomial experiment with $n$ trials and $k$
successes. 
Let $s_j$ count the number of failures that occur after the $j$th success, and
let  \[ t:= \sum_{j=1}^k s_j. \] 
This outcome will be represented in the right member of~\eqref{ncqbteq} by
the term $ q^t x^{n-k} y^k. $   We can think of $t = \log_q q^t$ as
a weighted count of failures that occur after successes.  

  Another equivalent way of thinking of $t$ is as the number of inversions in
a permutation of a sequence of $n-k$ $x$'s followed by $k$ $y$'s.
For example, consider the permutation $xxyxxyxy$ of $x^5 y^3 = xxxxxyyy$.
As shown earlier, $Q^{(3,1,0)} x^5 y^3 = xxyxxyxy$.  The inversions in a
permutation are pairs of a $y$ occurring before an $x$: these pairs are
in the third and fourth entry, the third and fifth entry, the third and seventh, 
and finally the sixth and seventh; four such inversions in all.  The number
of inversions in a given permutation corresponds to the size of the 
indexing partition; the size of $(3,1,0)$ is $4$.  
  The number of permutations of $x^{n-k} y^k$ containing exactly $t$ inversions
 is denoted $\inv(k, n-k; t)$, and by~\citet[p. 40, Theorem 3.5]{A76}, 
 \[ \inv(n-k, k ; t) = \#\Pp_{k, n-k}(t). \]
 Accordingly, those who prefer permutations to partitions may wish to
replace all
subsequent references to ``$\#\Pp_{k,n-k}(t)$'' by ``$\inv(n-k,k;t)$''.


If $S_j$ denotes the random variable that counts the number of failures after the $j$th 
success in a binomial experiment with $n$ independent
Bernoulli trials and probability of 
success equal to $\pi$, let 
\[ T:= \sum_{j=1}^k S_j. \]

\begin{equation} \label{pmfT} 
P(T=t) = \sum_{k=0}^n \#\Pp_{k,n-k}(t) (1-\pi)^{n-k} \pi^k, \end{equation} for 
$t=0,1,2,\dots, \lfloor n^2/4 \rfloor $;  and $0$, otherwise.
Note that the way we calculate $P(T=t)$ is to observe the exponent on $q$,
and then set $q=1$.  The order of these operation matters, because if we
set $q=1$ first, then the exponent on $q$ becomes inaccessible.

As an example, here is the previous table ($n=4$ case of the binomial
experiment) grouped by values $t$ of $T$, rather than
by the values of $Y$:
\begin{center}
\begin{tabular}{| c | c |c| c|}
\hline
 & & & corresponding\\
$t$ & Outcome & LHS summand & RHS summand\\
\hline\hline
       & FFFF & $xxxx$ & $ x^4$\\
       & FFFS & $xxxy$ & $ x^3 y $\\
$0$ & FFSS & $xxyy$ & $x^2 y^2  $\\
& FSSS & $xyyy$ & $xy^3$\\
& SSSS & $yyyy$ & $y^4$\\
\hline
& FFSF & $xxyx$ & $q x^3 y$\\
1 & FSFS & $xyxy$ & $q x^2 y^2 $\\
& SFSS & $yxyy$ & $q xy^3$ \\
\hline
& FSFF & $xyxx$ & $q^2 x^3 y $\\
2 &SFFS & $yxxy$ & $q^2 x^2 y^2 $\\
& FSSF & $xyyx$ & $q^2 x^2 y^2 $\\
& SSFS & $yyxy$ & $q^2 xy^3$ \\
\hline
& SFFF & $yxxx$ & $q^3 x^3 y$ \\
$3$ & SFSF & $yxyx$ & $q^3 x^2 y^2$  \\
& SSSF & $yyyx$ & $q^3 xy^3 $\\
\hline
$4$ & SSFF & $yyxx$ & $q^4 x^2 y^2$\\
 \hline \hline
\end{tabular}
\end{center}

\begin{remark}
Observe that for a given value $t$ of $T$, the summation bounds
\[ \sum_{k=0}^n \#\Pp_{k,n-k}(t) (1-\pi)^{n-k} \pi^k \] may include some terms
that are $0$.  For instance, in the preceding table we see that for $n = 4$, 
\[ P(T=3) = \sum_{k=1}^3 \#\Pp_{k,n-k}(t) (1-\pi)^{3-k}\pi^k \]  
since $\#\Pp_{0,4}(t)  = \#\Pp_{4,0}(t) = 0$.  
\end{remark}

The function~\eqref{pmfT} is a pmf since clearly $P(T=t) \geq 0$
for all $t$, and 
\begin{align*}
 \sum_{t=0}^{\lfloor n^2/4 \rfloor} P(T=t) &= 
 \sum_{t=0}^{\lfloor n^2/4 \rfloor} \sum_{k=0}^n \#\Pp_{k,n-k}(t) (1-\pi)^{n-k} 
 \pi^k\\
&= \sum_{k=0}^n  \gp{n}{k}{1} (1-\pi)^{n-k} \pi^k \\
& = \sum_{k=0}^n \binom{n}{k} (1-\pi)^{n-k} \pi^k \\
& = 1. 
 \end{align*}
 
\section{A refined binomial distribution}
To recap, for the experiment with $n$ independent Bernoulli trials 
and constant probability of 
success $\pi$ on each trial, let $Y$ count the number of successes in $n$ trials.
Then $Y\sim\Bin(n,\pi)$.   
In the last section, the random variable $T$ was defined as $T = \sum_{j=1}^k S_j$ 
where $S_j$ denotes the random variable that counts the number of failures that occur 
after the $j$th success.  Equivalently, $T$ counts the number of inversions 
(success before failure pairs) in the outcome of a given instance of the $n$
independent Bernoulli trials, each with probability of success $\pi$.  Next, we explore the
joint distribution of $Y$ and $T$.

\subsection{The joint distribution of $Y$ and $T$}
Having defined the random variables $Y$ and $T$, let us now consider their
joint distribution.

First, a table of the $n=4$ case displaying the outcomes and corresponding
probabilities for all values of $Y$ and $T$ in the support:
\vskip 1cm
\begin{tabular}{r|ccccc}
\hline \hline
    &     &    &  $T$ &  & \\
   $Y$ & 0 & 1 & 2 & 3 & 4 \\
\hline\hline
0 &  FFFF         & --- & --- & --- & --- \\
   &  $(1-\pi)^4$ & 0  & 0   & 0  &  0  \\
\hline
1  &  FFFS        & FFSF  & FSFF & SFFF & --- \\
    & $(1-\pi)^3 \pi$ & $(1-\pi)^3 \pi$ & $(1-\pi)^3 \pi$ & $(1-\pi)^3 \pi$ &  $0$ \\
\hline
   2 & FFSS  & FSFS & SFFS, FSSF & SFSF & SSFF \\
     & $(1-\pi)^2 \pi^2$ & $(1-\pi)^2 \pi^2$ & $2(1-\pi)^2 \pi^2$ & $(1-\pi)^2 \pi^2$ & $(1-
     \pi)^2 \pi^2$ \\ \hline 
 3 & FSSS & SFSS & SSFS & SSSF & --- \\
    & $(1-\pi)\pi^3$ & $(1-\pi)\pi^3$ & $(1-\pi)\pi^3$ & $(1-\pi)\pi^3$ &  $0$ \\
  \hline   
 4 & SSSS & --- &  --- &  --- &  --- \\
    & $\pi^4$ & 0 & 0 & 0 & 0\\
    \hline
\end{tabular}

\begin{align} \label{jointpmf}
P(Y=k, T=t) &= \left(\mbox{co\"eff of  $q^t$ in } \gp nkq \right)  (1-\pi)^{n-k} \pi^k \\
  & = \#\Pp_{k,n-k}(t) (1-\pi)^{n-k} \pi^k \notag
\end{align} 
for $k = 0, 1, 2, \dots, n$ and for each $k$, $t = 0, 1, \dots, k(n-k)$; and
$0$ otherwise.\\

It is immediate that $P(Y=k, T=t) \geq 0$ for all $k$ and $t$.  Also,
\begin{align*}
  \sum_{k=0}^n \sum_{t=0}^{k(n-k)} P(Y=k, T=t) &= 
    \sum_{k=0}^n \sum_{t=0}^{k(n-k)}  \#\Pp_{k,n-k}(t) (1-\pi)^{n-k} \pi^k  \\
    &=\sum_{k=0}^n  (1-\pi)^{n-k} \pi^k \sum_{t=0}^{k(n-k)} \#\Pp_{k,n-k}(t) \\
    & =\sum_{k=0}^n (1-\pi)^{n-k} \pi^k \binom nk\\
    &=1.
\end{align*}
Thus~\eqref{jointpmf} defines a joint pmf.  

The marginal pmf of $T$ is given in~\eqref{pmfT}, and the
marginal pmf of $Y\sim\Bin(n,\pi)$.

\subsection{A summation lemma}
In order to derive various moments,
we will
need
to take certain weighted sums of the co\"efficients of $\gp{n}{k}{q}$.
These will be proved in the following lemma.

\begin{lemma}  Let $n$ and $k$ be fixed nonnegative integers.
Then
 \begin{align}
  \label{lem1} \sum_{j\geq 0} \Big( j \cdot \#\Pp_{k,n-k}(j) \Big) 
 &= \binom{n}{2} \binom{n-2}{k-1}, \\
 \label{lem2}
 \sum_{j\geq 0} \left( j^2\cdot \# \Pp_{k,n-k}(j) \right) &=  \binom{n}{k} \frac{k(n-k)}{12}
\Big( n+1+3k(n-k) \Big). 
 \end{align}
   where we follow the
convention that $\binom{n}{-1} = \binom{n}{n+1} = 0$ for all $n$.
\end{lemma}

\begin{proof}
Recall that $\#\Pp_{k,n-k}(j)$ is given by the co\"efficient of $q^j$ in
$\gp{n}{k}{q}$.  The desired sum in~\eqref{lem1} is therefore equal to
\[ \frac{d}{dq} \gp{n}{k}{q} \Bigg\vert_{q=1}. \]
Letting 
\[ f(q) = \gp nkq = \prod_{j=1}^k \frac{1-q^{n-k+j}}{1-q^j}, \] and proceeding
by logarithmic differentiation:
\[ \log f(q) = \sum_{j=1}^k \Big( \log(1-q^{n-k+j}) - \log(1-q^j) \Big) , \] thus
\[ \frac{d}{dq} f(q) = \gp nkq \sum_{j=1}^k 
  \left( \frac{j q^j}{1-q^j} - \frac{(n-k+j) q^{n-k+j-1}}
{1-q^{n-k+j}}  \right). \]

To find $f'(1)$, put the $j$ term over a common denominator, apply
L'H\^{o}pital's rule to find the limit as $q\to 1$, and finally obtain
\begin{equation*} 
 f'(1) = \binom{n}{k} \sum_{j=1}^k \frac{n-k}{2} = \binom{n}{k} \frac{(n-k)k}{2} = 
\binom{n}{2}\binom{n-2}{k-1}. 
\end{equation*}

Next, the desired sum in~\eqref{lem2} is equal to
\[  \frac{d}{dq} \left( q \frac{d}{dq} \gp nkq \right) \Bigg | _{q=1}. \]
The details of the derivation are straightforward, as in~\eqref{lem1}, but even more
tedious to do by hand; and therefore omitted.
\end{proof}

\subsection{Various moments}

\begin{theorem} \label{ET} The first two raw moments of the random variable $T$ 
defined by the pmf~\eqref{pmfT} are
given by
   \[ E(T) = \binom{n}{2}\pi(1-\pi)\] and
   \[ E(T^2) = \binom{n}{2}\pi(1-\pi) \left( \frac{2n-1}{3} + \binom{n-2}{2} \pi(1-\pi)   \right),
   \]
and thus the variance
 \[ V(T) = \binom{n}{2}\pi(1-\pi) \left( \frac{2n-1}{3} - \pi(1-\pi)(2n-3) \right).   \]
\end{theorem}

\begin{proof}
\begin{align*}
E(T) &=  \sum_{t=0}^{\lfloor n^2/4 \rfloor} t \ P(T=t) \\ &= 
 \sum_{t=0}^{\lfloor n^2/4 \rfloor} \sum_{k=0}^n  t \#\Pp_{k,n-k}(t) (1-\pi)^{n-k} 
 \pi^k\\
& = \sum_{k=0}^n (1-\pi)^{n-k} \pi^k \sum_{t=0}^{k(n-k)} t\ \#\Pp_{k,n-k}(t) \\
& = \sum_{k=0}^n \binom{n}{2} \binom{n-2}{k-1} (1-\pi)^{n-k} \pi^k \mbox{  (by 
~\eqref{lem1})}\\
& = \binom{n}{2}\pi(1-\pi)\sum_{k=1}^{n-2} \binom{n-2}{k-1} (1-\pi)^{n-k-1}\pi^{k-1}\\
& =  \binom{n}{2} \pi(1-\pi).
 \end{align*}

\begin{align*}
E(T^2) &=  \sum_{t=0}^{\lfloor n^2/4 \rfloor } t^2 \ P(T=t) \\ &= 
 \sum_{t=0}^{\lfloor n^2/4 \rfloor } \sum_{k=0}^n  t^2 \#\Pp_{k,n-k}(t) (1-\pi)^{n-
k}  \pi^k\\
& = \sum_{k=0}^n  \binom{n}{k} \frac{k(n-k)}{12}
\Big( n+1+3k(n-k) \Big) (1-\pi)^{n-k} \pi^k \mbox{  (by 
~\eqref{lem2})}\\
& = \binom{n}{2}\pi(1-\pi) \left( \frac{2n-1}{3} + \binom{n-2}{2} \pi(1-\pi) \right),
 \end{align*}
where the last equality follows by 
hypergeometric summation~\cite[see, e.g.][]{PWZ96}.
 
 One can derive $V(T)$ in the usual way as $E(T^2) - \left( E(T) \right)^2$.  
\end{proof}

\begin{theorem} For the distribution defined by joint pmf~\eqref{jointpmf},
 \begin{align}
    E(YT) &= \binom n2 \pi(1-\pi)\Big( \pi(n-2) + 1\Big), \label{EYT}\\
    \Cov(Y,T) & = \binom n2 \pi(1-\pi) (1-2\pi)  \label{CovYT}.
 \end{align}
\end{theorem}
\begin{proof}
 \begin{align*}
 E(YT) &= \sum_{k=0}^n \sum_{t=0}^{k(n-k)} k t \#\Pp_{k,n-k}(t) (1-\pi)^{n-k} \pi^k \\
   & =  \sum_{k=0}^n k (1-\pi)^{n-k} \pi^k \sum_{t=0}^{k(n-k)} t \ \#\Pp_{k,n-k}(t) \\
   & = \sum_{k=0}^n k (1-\pi)^{n-k}  \pi^k
   \binom n2 \binom{n-2}{k-1} \mbox{ (by~\eqref{lem1})}\\
   & = \binom{n}{2} \pi(1-\pi) (1 - 2\pi + n\pi) \mbox{  (by hypergeometric summation)}.
 \end{align*}
 Then $\Cov(Y,T) = E(Y T) - E(Y) E(T).$
\end{proof}

\subsection{The conditional distribution of $T$ given $Y$}
Letting $n$ and $k$ be given, the pmf of the conditional distribution of $T$
given $Y=k$ is given by
\[ P(T = t\ |\ Y=k) = \frac{ \# \Pp_{k,n-k}(t) }{ \binom nk} 
= \binom nk^{-1} \cdot \mbox{ co\"eff of $q^t$ in} \gp nkq .\]
\begin{theorem} \label{CondMoments}
The conditional expectation \[ E(T\ | \ Y=k) = \frac{k(n-k)}{2} \] and
the conditional variance
\[ V(T\ | \ Y=k) = \frac{k(n-k)(n+1)}{12}. \]
\end{theorem}
\begin{proof}
\begin{align*} 
E(T\ |\ Y=k) &= \binom nk^{-1} \sum_{t=0}^{\lfloor n^2/4 \rfloor} t \cdot \#\Pp_{k,n-k}(t)\\
& = \binom nk^{-1} \frac{d}{dq} \gp nkq \Bigg|_{q=1} \\
& \overset{\mbox\small{~\eqref{lem1}}}= \binom nk^{-1} \binom n2 \binom{n-2}{k-1}
\\ &= \frac{k(n-k)}{2}.
\end{align*}
\begin{align*}
E(T^2\ |\ Y=k) &= \binom nk^{-1} \sum_{t=0}^{\lfloor n^2/4 \rfloor} t^2 \cdot
\#\Pp_{k,n-k}(t) \\
& \overset{~\eqref{lem2}}{=} \frac{k(n-k)}{12} \Big( n+1+3k(n-k) \Big).  
\end{align*}
\end{proof}

\subsection{The conditional distribution of $Y$ given $T$}
Letting $n$ and $t$ be given, the pmf of the conditional distribution of $Y$
given $T=t$ is given by
\begin{align*}
P(Y=k\ | \ T=t) & = \frac{ \#\Pp_{k,n-k}(t) \pi^k (1-\pi)^{n-k}  }
   { \sum_{j} \#\Pp_{j,n-j}(t) \pi^j (1-\pi)^{n-j} } \\
   & = \frac{ \#\Pp_{k,n-k}(t) \theta^k }
   { \sum_{j} \#\Pp_{j,n-j}(t) \theta^j },
 \end{align*}
where \begin{equation} \label{theta} \theta = \frac{\pi}{1-\pi} . \end{equation}
It is thus clear that the conditional distribution of $Y$ given $T=t$ is considerably
messier than the expressions encountered thus far.
For example, in the $n=4$ case, 
\begin{align*}
P(Y=2\ |\ T=3) &= \frac{(1-\pi)^2\pi^2}
{(1-\pi)^3\pi + (1-\pi)^2\pi^2 + (1-\pi)\pi^3} \\
& =  \frac{(1-\pi)\pi}{  (1-\pi)^2 + (1-\pi)\pi + \pi^2  }\\
& = \frac{ (1-\pi)\pi }{1-\pi + \pi^2},
\end{align*} where the algebraic simplifications undertaken here
do not generalize to 
arbitrary $n$, $k$, and $t$.

An expression for the $r$th raw moment of $Y$ given $T=t$,
\begin{align*}
E( Y^r | T=t) &= \sum_k  k^r \frac{\#\Pp_{k,n-k}(t) \pi^k (1-\pi)^{n-k} }
{ \sum_j  \#\Pp_{j,n-j}(t) \pi^j (1-\pi)^{n-j} } \\
& = \frac{ \sum_k k^r \cdot \#\Pp_{k,n-k}(t) \ \theta^k }
              { \sum_j \#\Pp_{j,n-j}(t) \  \theta^j },
\end{align*} with $\theta$ as in~\eqref{theta}, 
is therefore not particularly enlightening.
\subsection{Example/Application}
One way to think of our random variable $T$
is as a measure of homogeneity in the sense described as follows.  
Consider the classic example of tossing a fair coin $n=15$ times.
Suppose that $6$ of these tosses come up heads.  From the classical binomial
distribution we know that there are $\binom{15}{6} = 5005$ different $15$-tuples
that consist of $6$ heads and $9$ tails.  But if we further allow the $t$ 
to encode additional information about the where the $6$ heads and
$9$ tails appear in the tuple, we may observe the following.  The possible values of
$t$ in this example are $0$ through $54$ inclusive, because the degree of
the Gaussian polynomial $\gp nkq$ is $k(n-k) = 6(15-6) = 54$ and none of the
co\"efficents of $q$ in  
the Gaussian polynomial vanish.   By Theorem~\ref{CondMoments},
the mean value of $T$ given $6$ successes is $27$.  A value of $T$ close to $27$
indicates that the heads and tails are very well mixed together, i.e. are rather
homogeneous.  On the other hand, value of $T$ close to $0$ indicates that nearly 
all of the tails occurred up front, with nearly all of the heads near the end.  
A value of $T$ close to the maximum ($54$ in this example) indicate the opposite:
that nearly all of the heads occurred in the early trials and nearly all of the tails
occurred in the later trials.

\section{Conclusion}
In this paper, we consider interpreting the Potter--Sch\"utzenberger 
quantum binomial theorem, and a generalization thereof, as
providing a refinement of the binomial probability distribution, in 
which additional information
is preserved 
about the sequence of successes and failures in the in the underlying
binomial experiment.  

 It seems plausible that
other discrete probability distributions could be refined in an analogous way.
This possibility will be explored in future work.

\section*{Acknowledgments}
The author thanks George Andrews, 
Charles Champ, Broderick Oluyede, Robert Schneider, Divine
Wanduku, and Doron Zeilberger for encouragement, and helpful 
conversations and suggestions relating to this work.  Additionally, the author thanks 
Michael~\citet[p. 1]{S20}  for
drawing his attention to the work of H. S. A. Potter.  Finally, the author thanks
the anonymous referees for carefully reading the manuscript,
catching a substantial error in an earlier draft, and offering numerous helpful
suggestions.
 
\bibliographystyle{natbib-harv}

\end{document}